\newif \ifSIAM
\SIAMfalse

\ifSIAM
\documentclass[review]{siamart251216}
\newsiamremark{remark}{Remark}
\AddToHook{env/remark/begin}{\crefalias{theorem}{remark}}
\else
\documentclass[a4paper]{article}
\usepackage[margin=3.5cm]{geometry}
\usepackage{amsthm}
\usepackage{hyperref}
\hypersetup{
	colorlinks,
	linkcolor=blue,
	citecolor=blue,
	urlcolor=blue,
}
\theoremstyle{plain}%
\newtheorem{theorem}{Theorem}[section]%
\newtheorem{lemma}[theorem]{Lemma}%
\newtheorem{remark}[theorem]{Remark}%
\fi
\let\citet\cite

\usepackage{import}
\usepackage{amsmath}
\usepackage{mathtools}
\usepackage{amssymb}
\usepackage{mathrsfs}
\usepackage{xparse}

\makeatletter
\@ifpackageloaded{unicode-math}{%
	\newcommand{\mymathbold}{\symbf}%
}{%
	\usepackage{bm}%
	\newcommand{\mymathbold}{\bm}%
}
\makeatother

\newcommand{\scrbar}[1]{\overline{\mathcal{#1}}}
\newcommand{\scrhat}[1]{\widehat{\mathcal{#1}}}
\newcommand{\scrtl}[1]{\widetilde{\mathcal{#1}}}

\ExplSyntaxOn

\cs_new_protected:Nn \bamboo_define:nnnnnN
{
	\cs_new_protected:cpx { #3 #1 #4 } { \exp_not:N #5{#6{#2}} }
}

\int_step_inline:nnn { `A } { `Z }
{
	\bamboo_define:nnnnnN
	{ \char_generate:nn { #1 } { 11 } }
	{ \char_generate:nn { #1 } { 11 } }
	{ bl }
	{ }
	{ \mymathbold }
	\use:n
	
	\bamboo_define:nnnnnN
	{ \char_generate:nn { #1 } { 11 } }
	{ \char_generate:nn { #1 } { 11 } }
	{ scr }
	{ }
	{ \mathcal }
	\use:n
	
	\bamboo_define:nnnnnN
	{ \char_generate:nn { #1 } { 11 } }
	{ \char_generate:nn { #1 } { 11 } }
	{ }
	{ hat }
	{ \widehat }
	\use:n
	
	\bamboo_define:nnnnnN
	{ \char_generate:nn { #1 } { 11 } }
	{ \char_generate:nn { #1 } { 11 } }
	{ scr }
	{ hat }
	{ \scrhat }
	\use:n
	
	\bamboo_define:nnnnnN
	{ \char_generate:nn { #1 } { 11 } }
	{ \char_generate:nn { #1 } { 11 } }
	{ }
	{ br }
	{ \overline }
	\use:n
	
	\bamboo_define:nnnnnN
	{ \char_generate:nn { #1 } { 11 } }
	{ \char_generate:nn { #1 } { 11 } }
	{ }
	{ tl }
	{ \widetilde }
	\use:n
	
	\bamboo_define:nnnnnN
	{ \char_generate:nn { #1 } { 11 } }
	{ \char_generate:nn { #1 } { 11 } }
	{ scr }
	{ br }
	{ \scrbar }
	\use:n
	
	\bamboo_define:nnnnnN
	{ \char_generate:nn { #1 } { 11 } }
	{ \char_generate:nn { #1 } { 11 } }
	{ scr }
	{ tl }
	{ \scrtl }
	\use:n
	
	\bamboo_define:nnnnnN
	{ \char_generate:nn { #1 } { 11 } }
	{ \char_generate:nn { #1 } { 11 } }
	{ }
	{ dt }
	{ \dot}
	\use:n
}
\int_step_inline:nnn { `a } { `z }
{
	\bamboo_define:nnnnnN
	{ \char_generate:nn { #1 } { 11 } }
	{ \char_generate:nn { #1 } { 11 } }
	{ bl }
	{ }
	{ \mymathbold }
	\use:n
	
	\bamboo_define:nnnnnN
	{ \char_generate:nn { #1 } { 11 } }
	{ \char_generate:nn { #1 } { 11 } }
	{ }
	{ hat }
	{ \hat }
	\use:n
	
	\bamboo_define:nnnnnN
	{ \char_generate:nn { #1 } { 11 } }
	{ \char_generate:nn { #1 } { 11 } }
	{ }
	{ br }
	{ \bar }
	\use:n
	
	\bamboo_define:nnnnnN
	{ \char_generate:nn { #1 } { 11 } }
	{ \char_generate:nn { #1 } { 11 } }
	{ }
	{ tl }
	{ \tilde }
	\use:n                            
	
	\bamboo_define:nnnnnN
	{ \char_generate:nn { #1 } { 11 } }
	{ \char_generate:nn { #1 } { 11 } }
	{ }
	{ dt }
	{ \dot}
	\use:n
}
\clist_map_inline:nn
{
	Gamma,Delta,Theta,Lambda,Xi,Pi,Sigma,Phi,Psi,Omega
}
{
	\bamboo_define:nnnnnN
	{ #1 }
	{ #1 }
	{ bl }
	{ }
	{ \mymathbold }
	\use:c
	
	\bamboo_define:nnnnnN
	{ #1 }
	{ #1 }
	{ op }
	{ }
	{ \op }
	\use:c
	
	\bamboo_define:nnnnnN
	{ #1 }
	{ #1 }
	{ scr }
	{ }
	{ \mathcal }
	\use:c
	
	\bamboo_define:nnnnnN
	{ #1 }
	{ #1 }
	{ }
	{ hat }
	{ \widehat }
	\use:c
	
	\bamboo_define:nnnnnN
	{ #1 }
	{ #1 }
	{ }
	{ br }
	{ \overline }
	\use:c
	
	\bamboo_define:nnnnnN
	{ #1 }
	{ #1 }
	{ }
	{ tl }
	{ \widetilde }
	\use:c
	
	\bamboo_define:nnnnnN
	{ #1 }
	{ #1 }
	{ }
	{ dt }
	{ \dot}
	\use:c
	
}
\clist_map_inline:nn
{
	alpha,beta,gamma,delta,epsilon,zeta,eta,theta,iota,kappa,
	lambda,mu,nu,xi,pi,rho,sigma,tau,phi,chi,psi,omega
}
{
	\bamboo_define:nnnnnN
	{ #1 }
	{ #1 }
	{ bl }
	{ }
	{ \mymathbold }
	\use:c
	
	\bamboo_define:nnnnnN
	{ #1 }
	{ #1 }
	{ }
	{ hat }
	{ \hat }
	\use:c
	
	\bamboo_define:nnnnnN
	{ #1 }
	{ #1 }
	{ }
	{ br }
	{ \bar }
	\use:c
	
	\bamboo_define:nnnnnN
	{ #1 }
	{ #1 }
	{ }
	{ tl }
	{ \tilde }
	\use:c
	
	\bamboo_define:nnnnnN
	{ #1 }
	{ #1 }
	{ }
	{ dt }
	{ \dot}
	\use:c
}

\ExplSyntaxOff

\DeclareMathOperator{\E}{\mathbf{E}}

\newcommand{\rank}{\operatorname{rank}}

\newcommand{\tr}{\operatorname{tr}}

\DeclarePairedDelimiter{\norm}{\lVert}{\rVert}

\DeclarePairedDelimiter{\abs}{\lvert}{\rvert}

\DeclarePairedDelimiter{\parens}{(}{)}
\DeclarePairedDelimiter{\brackets}{[}{]}
\DeclarePairedDelimiterX{\ip}[2]{\langle}{\rangle}{#1,#2}

\DeclarePairedDelimiterXPP{\normsub}[2]{}{\lVert}{\rVert}{_{#2}}{#1}
\DeclarePairedDelimiterXPP{\ipsub}[3]{}{\langle}{\rangle}{_{#3}}{#1,#2}

\DeclarePairedDelimiterXPP{\ipHS}[2]{}{\langle}{\rangle}{_{\mathrm{HS}}}{#1, #2}
\DeclarePairedDelimiterXPP{\normHS}[1]{}{\lVert}{\rVert}{_{\mathrm{HS}}}{#1}

\DeclarePairedDelimiterXPP{\ipF}[2]{}{\langle}{\rangle}{_{\mathrm{F}}}{#1, #2}
\DeclarePairedDelimiterXPP{\normF}[1]{}{\lVert}{\rVert}{_{\mathrm{F}}}{#1}

\DeclarePairedDelimiterXPP{\dkl}[2]{\operatorname{D_{KL}}}{(}{)}{}{#1 \: \delimsize\Vert \: #2}

\DeclarePairedDelimiterXPP{\restr}[2]{}{{}}{\vert}{_{#2}}{#1}

\newcommand{\transpose}{^\top\! }

\newcommand{\R}{\mathbf{R}}

\newcommand{\range}{\operatorname{range}}

\newcommand{\negqquad}{\mspace{-36mu}}

\newcommand{\stquad}{\quad\text{s.t.}\quad}

\usepackage{bookmark}
\usepackage[capitalize]{cleveref}

\DeclarePairedDelimiterXPP{\opnorm}[1]{}{\lVert}{\rVert}{_{\mathrm{op}}}{#1}
\DeclarePairedDelimiterXPP{\nucnorm}[1]{}{\lVert}{\rVert}{_{\mathrm{nuc}}}{#1}

\newcommand{\rgt}{r_*}
\newcommand{\ropt}{r}
\newcommand{\Mst}{M_*}
\newcommand{\Ust}{U_*}
\newcommand{\Vst}{V_*}
\newcommand{\Wst}{W_*}
\newcommand{\Ubal}{U_{\mathrm{bal}}}
\newcommand{\Vbal}{V_{\mathrm{bal}}}
\newcommand{\Ubaldt}{\Udt_{\mathrm{bal}}}
\newcommand{\Vbaldt}{\Vdt_{\mathrm{bal}}}

\newcommand{\symms}{\mathbf{S}}

\newcommand{\Tst}{\scrT}
\newcommand{\PTst}{\scrP_{\Tst}}

\newcommand{\symspace}{\symms_d}
\newcommand{\psdspace}{\symms_d^+}
\newcommand{\Uspace}{\R^{d \times \ropt}}
\newcommand{\Mspace}{\R^{d_1 \times d_2}}
\newcommand{\Mspaced}{\R^{d \times d}}

\newcommand{\Lspace}{\R^{d_1 \times \ropt}}
\newcommand{\Rspace}{\R^{d_2 \times \ropt}}

\newcommand{\Lstspace}{\R^{d_1 \times \rgt}}
\newcommand{\Rstspace}{\R^{d_2 \times \rgt}}

\newcommand{\longtitle}{Low-rank matrix recovery landscapes beyond RIP\\with application to rank-one measurements}
\newcommand{\shorttitle}{Low-rank matrix recovery landscapes beyond RIP}
\newcommand{\fundingack}{This work was supported by Hi!~PARIS and the ANR/France 2030 program (ANR-23-IACL-0005).}
\newcommand{\myaddress}{CERMICS, ENPC, Institut Polytechnique de Paris,	CNRS, Marne-la-Vallée, France}
\newcommand{\myemail}{andrew.mcrae@enpc.fr}

\ifSIAM
\title{\longtitle%
	\thanks{Submitted to the editors September 5, 2026.%
	\funding{\fundingack{}}}}

\author{Andrew D.\ McRae%
	\thanks{\myaddress{} (\email{\myemail}).}%
}

\headers{\shorttitle}{A. D. McRae}

\else
\title{\longtitle}
\author{Andrew D.\ McRae\thanks{\myaddress{} (\texttt{\myemail}). \fundingack{}}}
\fi

\begin{document}
\maketitle
\begin{abstract}
	We study the problem of low-rank matrix recovery from linear measurements via the global nonconvex landscape of a low-rank factored formulation of the matrix LASSO (nuclear-norm--regularized least-squares).
	If the landscape is benign, that is, has no bad local optima, then practical and scalable algorithms can compute good statistical estimates.
	Previous state-of-the-art landscape guarantees have typically assumed that the linear measurement operator has the restricted isometry property,
	that is, the operator is approximately an isometry over all low-rank matrices.
	This is an unrealistic assumption for many applications; in particular, when the individual measurement matrices are themselves low-rank,
	we typically have poor upper isometry constants.
	To overcome this, we establish new guarantees of a benign landscape under a weaker isometry condition:
	rather than requiring upper isometry over all low-rank matrices,
	we only require it over the linear low-rank tangent space to the low-rank ground truth matrix.
	To illustrate the utility of this result, we apply it to the problem of matrix recovery from random rank-one linear measurements;
	via high-probability concentration bounds on the random measurement operator,
	we prove a novel landscape guarantee 
	with statistically near-optimal sample complexity and recovery error.
\end{abstract}

\newcommand{\errbdcomps}{[ \sqrt{\rgt} \lambda  + \sqrt{\ropt + \rgt}(\opnorm{\scrA^*(\xi)} - \lambda)_+ ]}
\newcommand{\y}{y}
\section{Introduction and main results}
\label{sec:intro}
We consider the problem of estimating a low-rank matrix $\Mst \in \Mspace$
from $n$ potentially noisy linear measurements of the form
\begin{equation*}
	\y_i = \ip{A_i}{\Mst} + \xi_i, \quad i = 1, \dots, n,
\end{equation*}
where $A_1, \dots, A_n \in \Mspace$ are known matrices, $\ip{\cdot}{\cdot}$ denotes the standard (trace) inner product on matrices,
and $\xi_1, \dots, \xi_n$ represent measurement (and/or modeling) error.
Letting $\scrA \colon \Mspace \to \R^n$ be the linear operator defined by $\scrA(M)_i = \ip{A_i}{M}$,
we can write our measurement model more compactly as
\begin{equation}
	\label{eq:ms_model}
	\y = \scrA(\Mst) + \xi \in \R^n.
\end{equation}
For the purpose of practically computing an estimate to the unknown $\Mst$ from the data $(\scrA, \y)$,
we study the following regularized factored low-rank least-squares problem:
\begin{equation}
	\min_{\substack{U\in\Lspace \\ V\in\Rspace}}~\frac{1}{2}\norm{\scrA(UV\transpose)-\y}^{2} + \lambda \frac{\normF{U}^2 + \normF{V}^2}{2}, \label{eq:lasso_fact}
\end{equation}
where $\norm{\cdot}$ denotes the ordinary Euclidean norm on $\R^n$,
$\normF{\cdot}$ denotes the Frobenius (elementwise Euclidean) norm on matrices,
$r \geq 1$ is an integer search rank hyperparameter, and $\lambda \geq 0$ is a regularization hyperparameter.

The nonconvex optimization problem \eqref{eq:lasso_fact} is a popular practical choice for low-rank matrix recovery; conceptually, this arises in two ways often separate in the literature:
\begin{enumerate}
	\item In the case $\lambda > 0$, it is a more computationally scalable alternative to the \emph{matrix LASSO}, which is the problem
	\begin{equation}
		\label{eq:lasso}
		\min_{M\in\Mspace}~\frac{1}{2}\norm{\scrA(M)-\y}^{2}+\lambda\nucnorm M,
	\end{equation}
	where $\nucnorm{\cdot}$ denotes the matrix nuclear norm.
	This is popular due to its convexity,
	and there is a rich literature of statistical guarantees (many of them optimal) under a wide variety of measurement types.
	However, for large problem sizes, \eqref{eq:lasso} is computationally challenging to solve directly via convex programming; hence \eqref{eq:lasso_fact} is often preferred as a heuristic.
	See the survey \cite{Davenport2016} for an overview and many historical references.

	\item In the case $\lambda = 0$, the problem \eqref{eq:lasso_fact} is a smooth reparametrization of the rank-constrained matrix least-squares problem
	\begin{equation*}
		\min_{M \in \Mspace}~\frac{1}{2} \norm{\scrA(M) - \y}^2 \stquad \rank(M) \leq r.
	\end{equation*}
	Without other regularization (such as the nuclear norm in \eqref{eq:lasso}),
	the rank parameter controls the statistical model complexity and is essential for identifiability whenever $n < d_1 d_2$.
	See the surveys \cite{Davenport2016,Chi2019} for further background.
\end{enumerate}
More discussion of the relative strengths of each perspective can be found in the recent work \cite{McRae2026preprint}.

For algorithmic purposes, the principal difficulty of \eqref{eq:lasso_fact} is its nonconvexity;
standard local descent methods could be trapped by spurious local optima.
There is a wide variety of theoretical and algorithmic approaches to resolve the problem of spurious local optima.
See, again, \cite{Chi2019} for an overview; we discuss in \Cref{sec:intro_ms} below the work which is most relevant to our setting.

In this paper, we focus on the nonconvex \emph{landscape} of \eqref{eq:lasso_fact};
we want to show that it is \emph{benign}.
That is, we seek to show that, under certain assumptions on the model \eqref{eq:ms_model},
the problem \eqref{eq:lasso_fact} has no spurious local optima
or at least, from a statistical perspective, that all local optima yield good estimates of $\Mst$.
For many problems, this type of analysis has yielded theoretical guarantees under (near-)optimal statistical assumptions.

The majority of landscape analysis for \eqref{eq:lasso_fact} in the literature assumes that $\scrA$ has the \emph{restricted isometry property} (RIP).
A historical sampling of this work is \cite{Bhojanapalli2016,Park2017,Ge2017,Li2019,Ha2020,Zhu2021a,Zhang2025a,Ouyang2025}.
For an integer $k \geq 1$ and scalars $\beta \geq \alpha > 0$, we say that $\scrA$ has $(k, \alpha, \beta)$-RIP%
\footnote{The usual definition of RIP sets $\alpha = 1 - \delta$ and $\beta = 1 + \delta$ for some $\delta \in [0, 1)$. This is equivalent up to rescaling of $\scrA$.}
if, for all matrices $H \in \Mspace$ with rank at most $k$,
\begin{equation*}
	\alpha \normF{H}^2 \leq \norm{\scrA(H)}^2 \leq \beta \normF{H}^2.
	\tag{RIP}
	\label{eq:RIP}
\end{equation*}
Under this assumption, a state-of-the-art landscape guarantee for statistical recovery is the following,
which was partially shown in \cite{Zhang2025a} and recently further developed in \cite[Thms.~2~\&~3]{McRae2026preprint}:
Under the model \eqref{eq:ms_model},
with optimization rank $r \geq \rgt \coloneqq \rank(\Mst)$,
if $\scrA$ has $(r + \rgt, \alpha, \beta)$-RIP with
\begin{equation}
	\label{eq:cond_RIP}
	\frac{\beta}{\alpha} < 1 + 2 \sqrt{\frac{r}{\rgt}},
\end{equation}
then any second-order critical point%
\footnote{A second-order critical point is one at which the gradient is zero and the Hessian is positive semidefinite; see \Cref{sec:proofs_determ} for explicit expressions for our problem. This is a weaker assumption than local optimality, and there are various guarantees in the optimization literature (e.g., \cite{Cartis2012,Lee2019b}) that standard algorithms will find such points.}
$(U, V)$ of \eqref{eq:lasso_fact} satisfies, with $M \coloneqq U V\transpose$,
\begin{equation}
	\label{eq:err_RIP}
	\normF{M - \Mst} \leq C_{\alpha, \beta} \errbdcomps{},
\end{equation}
where
$C_{\alpha, \beta} > 0$ depends only on $\alpha$ and $\beta$,
$\scrA^* \colon \R^n \to \Mspace$ is the adjoint of $\scrA$,
$\opnorm{\cdot}$ denotes the matrix ($\ell_2$) operator norm,
and, for $x \in \R$, we denote $x_+ \coloneqq \max\{x, 0\}$.
As discussed in \cite{McRae2026preprint}, the error bound \eqref{eq:err_RIP} is optimal within a multiplicative constant
and thus reveals the effects of the measurement error $\xi$ and the hyperparameters $r$ and $\lambda$ on the recovery error.

Although convenient for theory, RIP is a strong assumption for many applications.
It has been shown to hold, for example, when $A_1, \dots, A_n$ are dense, unstructured random matrices (e.g., the entries are standard Gaussians \cite{Candes2011b}),
but such measurements are burdensome to generate, store, and use.
For many more practical measurements, RIP does not hold.

Our main example problem will be the (asymmetric) rank-1 model
\begin{equation}
	\label{eq:r1_model}
	A_i = a_i b_i\transpose,
\end{equation}
where $a_1, \dots, a_n \in \R^{d_1}$ and $b_1, \dots, b_n \in \R^{d_2}$ are known vectors.
These measurements are more practical than the dense matrices typically considered for showing RIP.
This model encompasses problems such as matrix completion and blind deconvolution (see \Cref{rem:mc} below for some discussion and references), though these will not be our main focus.

If the vectors $\{a_i, b_i\}$ are chosen randomly, independently, and, for example, Gaussian,
then convex programming methods similar to \eqref{eq:lasso} have been shown to recover $\Mst$ with statistically optimal sample complexity and error \cite{Cai2015,Zhong2015}.
However, existing guarantees for more computationally scalable approaches like \eqref{eq:lasso_fact} are weaker; see \Cref{rem:ms_complexity} below for some discussion and comparison to prior work.

To show a benign landscape of \eqref{eq:lasso_fact},
results assuming RIP such as \cite{Zhang2025a,McRae2026preprint} do not apply without severe statistical suboptimality.
As noted by \cite{Cai2015,Zhong2015}, rank-1 measurements do not satisfy RIP (with useful values of $\alpha, \beta$) without an unreasonably large number $n$ of measurements
(consider the application of $\scrA$ to any of the rank-1 matrices $A_i$).

Without RIP, what can we say about the landscape of \eqref{eq:lasso_fact}?
There are relatively few results with rank-1 measurements.
The most relevant work we are aware of is \cite{Ge2016,Ge2017} for matrix completion.
However, these results require a statistically suboptimal number of samples and add additional problem-structure--specific regularization (see \Cref{rem:mc} below for more discussion).
Another work is \cite{McRae2026a}, which, building on prior work for phase retrieval (e.g., \cite{Sun2018}),
studies the landscape of a similar-looking problem: that of recovering a \emph{symmetric, positive definite} matrix $\Mst$ with measurements of the form $A_i = a_i a_i\transpose$.
However, their analysis depends critically on the fact that the ground truth and the measurement matrices are positive semidefinite, and thus we cannot adapt their results to the asymmetric model \eqref{eq:r1_model}. 
Thus new landscape tools are needed.

\subsection{Main deterministic landscape guarantee}
We now present the main result of this paper,
which is a landscape guarantee for \eqref{eq:lasso_fact} based on new deterministic assumptions on $\scrA$.

For completeness, our result includes recovery of positive semidefinite (PSD) symmetric matrices.
We denote by $\symspace$ the set of symmetric matrices in $\Mspaced$,
and we denote by $\psdspace \subseteq \symspace$ the set of PSD such matrices.
Thus, in the model \eqref{eq:ms_model},
we further assume that $\Mst \in \psdspace$ and $A_1, \dots, A_n \in \symspace$.
The symmetric PSD version of \eqref{eq:lasso_fact} is
\begin{equation}
	\min_{U \in \Uspace}~\frac{1}{2}\norm{\scrA(U U\transpose) - \y}^2 + \lambda \normF{U}^2 \label{eq:lasso_sym_fact}.
\end{equation}
Note that, unlike \cite{McRae2026a}, we do not require the measurement matrices $A_i$ to be PSD.

To state our result, we need an additional key concept in low-rank matrix geometry.
The \emph{tangent space}
to $\Mst$ is the linear subspace
\[
	\Tst \coloneqq \{ L \Mst + \Mst R\transpose : L \in \R^{d_1 \times d_1}, R \in \R^{d_2 \times d_2} \} \subseteq \Mspace.
\]
In the symmetric case (where $d_1 = d_2 = d$), we furthermore need $L = R$ so that $\Tst \subseteq \symspace$.
If $\rgt = \rank(\Mst)$, all matrices in $\Tst$ have rank at most $2 \rgt$.

\newcommand{\socptext}{for every $\lambda \geq 0$, every second-order critical point $(U, V)$ of \eqref{eq:lasso_fact} with $M \coloneqq U V\transpose$ (resp.\ every second-order critical point $U$ of \eqref{eq:lasso_sym_fact} with $M \coloneqq U U\transpose$) satisfies}

With this, we can state our main result:
\begin{theorem}
	\label{thm:determ}
	In the asymmetric (resp.\ symmetric) case, under the model \eqref{eq:ms_model},
	let $\ropt \geq \rgt \coloneqq \rank(\Mst)$, and suppose $\scrA$ satisfies the following for some $\beta \geq \alpha > 0$:
	\begin{enumerate}
		\item (Restricted lower isometry on error) For all $M \in \Mspace$ (resp.\ $M \in \psdspace$) with rank at most $\ropt$,
		\[
			\norm{\scrA(M - \Mst)}^2 \geq \alpha \normF{M - \Mst}^2.
		\]
		\item (Tangent space upper isometry) For all $H \in \Tst$,
		\[
			\norm{\scrA(H)}^2 \leq \beta \normF{H}^2.
		\]
	\end{enumerate}
	If
	\begin{equation}
		\label{eq:ab_cond}
		\frac{\beta}{\alpha} < \frac{6}{\sqrt{5} + 2} \approx 1.42,
	\end{equation}
	then, \socptext{}
	\begin{equation}
		\label{eq:errbd}
		\normF{M - \Mst} \leq \frac{(36\lambda + 12 \opnorm{\scrA^*(\xi)})\sqrt{\rgt} + 10\sqrt{\ropt+\rgt}(\opnorm{\scrA^*(\xi)} - \lambda)_+}{6\alpha - (\sqrt{5}+2)\beta }.
	\end{equation}
\end{theorem}
We prove this in \Cref{sec:proofs_determ} below.
The result is also valid in the complex case with the appropriate trivial adjustments in notation; the exact proof applies with these adjustments.
In the complex case, we understand second order criticality in the usual (real-variable) sense in terms of the real and complex parts of the variables.

\begin{remark}
	We do not guarantee recovery of the \emph{global optimum} (except in the case $\lambda = 0$, $\xi = 0$, when we obtain exact recovery of $\Mst$, and this is clearly a global optimum of \eqref{eq:lasso_fact}).
	To show global optimality with our framework would require assumptions on the rank and tangent space of such a global optimum similar to what we have assumed for the ground truth;
	for most applications with random measurements, showing that these hold for a global optimum is more difficult than for the ground truth because the optimum itself depends on $\scrA$.
	See \cite[Sec.~3.3]{McRae2026preprint} for further related discussion.
\end{remark}

\begin{remark}
\Cref{thm:determ} has several important similarities to and differences from results that assume RIP:
\begin{itemize}
	\item The definition of lower isometry constant $\alpha > 0$ in \Cref{thm:determ} is similar to that of \eqref{eq:RIP} except that it only needs to apply to matrices of the form $H = M - \Mst$;
	this could make a difference in some applications,
	but it is not a critical distinction, as one can show slightly weaker variants of \cite[Thm.~3]{McRae2026preprint} that only require our definition (see, in particular, that work's Remark~3).
	\item The key novelty in our assumptions is that the upper isometry constant $\beta > 0$ is only defined over $H$ in the tangent space $\Tst$ to $\Mst$,
	whereas \eqref{eq:RIP} requires upper isometry over all $H$ of rank up to $\ropt + \rgt$
	(a nuance: again by \cite[Rmk.~3]{McRae2026preprint}, there is a variant of the RIP result only requiring upper isometry on arbitrary $H$ of rank 2; however this does not help for our applications with rank-1 measurements).
	\item Our requirement \eqref{eq:ab_cond} plays a similar role to \eqref{eq:cond_RIP} for RIP,
	but the condition \eqref{eq:cond_RIP} allows a wider range of $\alpha, \beta$ and,
	notably, becomes weaker for larger values of the optimization rank $r$
	(thus showing a potential benefit to rank \emph{overparametrization}, with the caveat that $\alpha,\beta$ could also change with $r$).
	It is not clear whether we could show a similar effect of overparametrization under our assumptions:
	our weaker definition of upper isometry requires a fundamentally different proof technique than that of \cite{Zhang2025a,McRae2026preprint}.
	Relatedly, the work \cite{McRae2026a} also shows benefits of overparametrization under conditions qualitatively similar to ours (see, in particular, their Theorem~2),
	but their analysis depends critically on their (positive semidefinite) problem structure.
	
	\item The error bound \eqref{eq:errbd} is comparable to the bound \eqref{eq:err_RIP} from \cite{McRae2026preprint} and thus, within constants, optimal.
\end{itemize}
\end{remark}
To see the practical benefits of our weaker isometry assumption, we next return to our example application.

\subsection{Application to random rank-1 measurements}
\label{sec:intro_ms}
We again consider the rank-1 measurements \eqref{eq:r1_model}.
We will assume that the vectors $\{a_i, b_i\}_i$ are random, isotropic and sub-Gaussian:
a zero-mean random vector $X \in \R^D$ is
\begin{enumerate}
	\item \emph{Isotropic} if $\E X X\transpose = I_D$, or, equivalently, $\E \abs{\ip{X}{h}}^2 = \norm{h}^2$ for all fixed $h \in \R^D$; and
	\item $K$--\emph{sub-Gaussian} for some $K > 0$ if, for all fixed $h \in \R^D$ with $\norm{h} = 1$, $\E e^{\abs{\ip{X}{h}}^2/K^2} \leq 2$.
\end{enumerate}
This is a slightly more general assumption than is typical in the literature on matrix recovery from random rank-1 measurements (see \Cref{rem:ms_complexity} below);
the comparable works that we are aware of assume that $a_i, b_i$ are Gaussian or, in the case of \cite{Cai2015}, have independent scalar sub-Gaussian entries.

With these measurements, we have the following landscape result:
\begin{theorem}
	\label{thm:ms}
	Consider the model \eqref{eq:ms_model} with rank-1 measurements of the form \eqref{eq:r1_model}.
	Suppose $a_1, \dots, a_n \in \R^{d_1}$ and $b_1, \dots, b_n \in \R^{d_2}$ are independent random vectors, where $a_1, \dots, a_n$ are copies of a zero-mean, isotropic, $K_a$--sub-Gaussian random vector $a \in \R^{d_1}$,
	and $b_1, \dots, b_n$ are copies of a zero-mean, isotropic, $K_b$--sub-Gaussian random vector $b \in \R^{d_2}$.

	There exist universal constants $C_1, C_2, C_3 > 0$ such that, for any optimization rank parameter $\ropt \geq \rgt$, if
	\[
		n \geq C_1 K_a^4 K_b^4 \ropt(d_1 + d_2) \log (d_1 + d_2),
	\]
	then, with probability at least $1 - \frac{C_2}{d_1 + d_2}$,
	for any $\lambda \geq 0$,
	every second-order critical point $(U, V)$ of \eqref{eq:lasso_fact} satisfies, with $M \coloneqq U V\transpose$,
	\[
		\normF{M - \Mst} \leq \frac{C_3}{n} \errbdcomps{}.
	\]
\end{theorem}
This follows from a result for more general random measurements, \Cref{thm:ms_gen}, in \Cref{sec:ms} below.
It is likely that some version of this result is also true in the complex case (see the comment after \Cref{thm:determ}), but, as many of the probability tools we use are only stated in the real case in the literature,
we do not attempt to verify this in the present work.
The $O((d_1 + d_2)^{-1})$ failure probability bound could be improved with different (and perhaps less general) proof techniques.

\begin{remark}
	To our knowledge, this is the first benign landscape guarantee for random sub-Gaussian rank-1 measurements.
	The only existing landscape analysis we are aware of is \cite{Diaz2019},
	which studies (first-order) critical points of a nonsmooth least absolute deviations loss in the case $r = \rgt = 1$ (a version of blind deconvolution---see \Cref{rem:mc} below).
\end{remark}

\begin{remark}
	\label{rem:ms_complexity}
	When $r \approx \rgt$, the requirement that $n \gtrsim \rgt(d_1 + d_2) \log (d_1 + d_2)$ is, within constants and a logarithmic factor, optimal, as the set of rank-$\rgt$ matrices in $\Mspace$ has $\approx \rgt(d_1 + d_2)$ degrees of freedom.
	Recovery with the optimal sample complexity $n \approx \rgt(d_1 + d_2)$ was shown by \cite{Cai2015,Zhong2015} for a convex programming approach similar to \eqref{eq:lasso}.
	For more computationally scalable methods, previous recovery guarantees are less satisfying.
	A survey of several classes of algorithms and their associated guarantees is provided by \cite[Sec.~1.4]{Bahmani2021};
	that work proposes and analyzes another (unconventional) algorithm.
	One can also directly analyze more typical algorithms such as spectral initialization followed by gradient descent on \eqref{eq:lasso_fact} \cite{Zhong2015,Li2021a}
	(see also the similar results for matrix completion and blind deconvolution---some are referenced below in \Cref{rem:mc}).
	These results, however, typically require exact knowledge of the true rank $\rgt$ and a number of samples $n$ that scales suboptimally with $\rgt$ and the ground truth condition number.
	The work \cite{Foucart2019} does give recovery guarantees with the optimal sample complexity $n \gtrsim \rgt(d_1 + d_2)$ for a low-rank algorithm;
	however, the algorithm (a modified version of iterative singular value hard thresholding with a large, constant degree of rank overparametrization) is quite unconventional and requires iteratively computing truncated singular value decompositions of large matrices.
	
\end{remark}

\begin{remark}
	\label{rem:sym}
	One might wonder whether \Cref{thm:ms} could be extended to the symmetric case with $a_i = b_i$ as in, for example, \cite{Chen2015d,Kueng2017}.
	An extension of \Cref{thm:ms} to this case would be an interesting complement to results in \cite{McRae2026a}.
	However, our analysis in the present work depends critically on the fact that the sensing operator is, in expectation, an isometry,
	which is not the case with measurements of the form $A_i = a_i a_i\transpose$.
\end{remark}

\begin{remark}
	\label{rem:mc}
	Another seemingly promising application for our theory is matrix completion,
	where $A_i = e_{k_i} e_{\ell_i}\transpose$ for randomly distributed matrix indices $\{(k_i, \ell_i)\}_i$, where $e_k$ denotes the $k$th standard Euclidean basis vector in either $\R^{d_1}$ or $\R^{d_2}$.
	For this problem, tangent space isometry can be shown under the assumption that the row and column spaces of $\Mst$ are \emph{incoherent} with respect to the standard Euclidean basis (see, e.g., \cite[Ch.~4]{Wright2022} for an overview and many historical references).
	The work \cite{Ge2017} showed (under some strong statistical assumptions) benign landscape of a variant of \eqref{eq:lasso_fact} with an additional regularizer to promote incoherence.
	Unfortunately, without a guarantee of incoherence for $M$, our \emph{lower} isometry assumption in \Cref{thm:determ} does not hold for this measurement model, as demonstrated by a class of counterexamples given by \cite{Krahmer2021}.
	Treating this problem in our framework would likely require some guarantee of incoherence of matrices $M$ arising from second-order critical points,
	either (a) via an explicit regularizer as in \cite{Ge2017} or (b) by showing some kind of ``implicit regularization'' of second-order critical points
	(as is shown, again under strong statistical assumptions, for iterates of gradient descent in works such as \cite{Ma2020}).
	
	Similar considerations hold for the problem of blind deconvolution;
	see \cite{Krahmer2021,Ma2020,Chen2023} for more background and details. 
\end{remark}

\section{Deterministic proof}
\label{sec:proofs_determ}
This section is dedicated to proving \Cref{thm:determ}.
Although we mostly use the notation from the asymmetric case,
the proof is also valid for the symmetric case (setting $d_1 = d_2 = d$, $U = V$, etc.), with many steps becoming redundant.
As discussed after \Cref{thm:determ}, the result and its proof are also valid in the complex case with superficial modifications (replacing transpose by conjugate transpose, symmetric by Hermitian, inner products by their real part, etc.).

By standard calculations (see, e.g., \cite[Sec.~4]{McRae2026preprint}),
$(U, V)$ is a second-order critical point of \eqref{eq:lasso_fact} if and only if
\begin{equation}
	\begin{aligned}
		\scrA^*(\scrA(M) - \y) V + \lambda U &= 0, \\
		\scrA^*(\scrA(M) - \y)\transpose U + \lambda V &= 0
	\end{aligned}
	\label{eq:gradzero}
\end{equation}
and, for all $\Udt \in \Lspace$, $\Vdt \in \Rspace$,
\begin{equation}
	0 \leq \ip{\scrA^*(\scrA(M) - \y)}{ \Udt \Vdt \transpose }
	+ \frac{1}{2} \norm{\scrA( U \Vdt\transpose + \Udt V\transpose )}^2
	+ \frac{\lambda}{2} ( \normF{\Udt}^2 + \normF{\Vdt}^2).
	\label{eq:hessineq}
\end{equation}

Some additional terminology will be used several times in this section.
If $X$ is a rank-$k$ matrix with singular value decomposition $\Ubr \Sigma \Vbr\transpose$, where $\Ubr, \Vbr$ each have $k$ orthonormal columns,
and $\Sigma$ is $k \times k$ and diagonal with strictly positive diagonal entries,
then
\begin{enumerate}
	\item the \emph{polar factor} of $X$ is the matrix $\Ubr \Vbr\transpose$, and
	\item the Moore-Penrose pseudoinverse of $X$ is the matrix $X^\dagger = \Vbr \Sigma^{-1} \Ubr\transpose$.
\end{enumerate}

\subsection{Reduction to the balanced case}
In the asymmetric case,
a major challenge in the analysis of \eqref{eq:lasso_fact} is the possible imbalance of $(U, V)$.
We say that a pair $(U, V)$ is \emph{balanced} if $U\transpose U = V\transpose V$.
The following lemma says that, for the purpose of landscape analysis, we can assume without loss of generality that second-order critical points are balanced:
\begin{lemma}
	\label{lem:balred}
	Suppose $(U, V)$ is a second-order critical point of \eqref{eq:lasso_fact}, and set $M \coloneqq U V\transpose$.
	Then any balanced factorization $(\Ubal, \Vbal)$ of $M$ with $\Ubal \in \Lspace$, $\Vbal \in \Rspace$, $\Ubal \Vbal\transpose = M$, and $\Ubal\transpose \Ubal = \Vbal\transpose \Vbal$
	is also a second-order critical point of \eqref{eq:lasso_fact}.
\end{lemma}
This result also holds if we replace the quadratic function $M \mapsto \frac{1}{2} \norm{\scrA(M) - \y}^2$ by an arbitrary twice-differentiable convex function (as in, e.g., \cite{McRae2026preprint}).
\begin{proof}
	If $\lambda > 0$, the pair $(U, V)$ must, in fact, be balanced.
	Indeed, by \eqref{eq:gradzero}, we have
	\[
		\lambda U\transpose U = - U\transpose \scrA^*(\scrA(M) - \y) V = \lambda V\transpose V.
	\]
	To our knowledge, this fact was first noted by \cite{Li2019}.
	A balanced decomposition $M = U V\transpose$ is unique up to trivial right-multiplication of $U$ and $V$ by an orthogonal matrix,
	so this completes the proof in the case $\lambda > 0$.
	Thus it remains to consider the case $\lambda = 0$.
	
	As $\range(\Ubal) \subseteq \range(U)$ and $\range(\Vbal) \subseteq \range(V)$,
	we have
	\begin{equation}
		U U^\dagger \Ubal = \Ubal \quad \text{and} \quad V V^\dagger \Vbal = \Vbal,
		\label{eq:dag_id}
	\end{equation}
	and, furthermore, by \eqref{eq:gradzero} (with $\lambda = 0$), we have the balanced counterpart to \eqref{eq:gradzero}, which is
	\begin{align*}
		\scrA^*(\scrA(M) - \y) \Vbal &= 0, \quad \text{and} \\
		\scrA^*(\scrA(M) - \y)\transpose \Ubal &= 0.
	\end{align*}	
	It remains to show the Hessian inequality (with $\lambda = 0$), that is,
	\begin{equation}
		\label{eq:bal_hess}
		0 \leq \ip{\scrA^*(\scrA(M) - \y)}{ \Ubaldt \Vbaldt \transpose }
		+ \frac{1}{2} \norm{\scrA( \Ubal \Vbaldt\transpose + \Ubaldt \Vbal\transpose )}^2
	\end{equation}
	for all $\Ubaldt \in \Lspace$, $\Vbaldt \in \Rspace$.
	We consider two cases based on the rank of $M = U V\transpose$.
	If $(U, V)$ is a second-order critical point, and $\rank(M) < \ropt$, then, by \cite[Thm.~2.3]{Ha2020}, we must have the stronger condition $\scrA^*(\scrA(M) - \y) = 0$.
	
	In the remaining case where $\rank(M) = \ropt$,
	then $U$, $V$, $\Ubal$, and $\Vbal$ must all have full rank $\ropt$,
	so
	\[
		U V\transpose = M = \Ubal \Vbal\transpose \quad \Longrightarrow \quad U^\dagger \Ubal (V^\dagger \Vbal)\transpose = U^\dagger U V\transpose (V\transpose)^\dagger = I_{\ropt}.
	\]
	This furthermore implies that
	\begin{equation}
		(V^\dagger \Vbal)\transpose U^\dagger \Ubal = I_{\ropt}.
		\label{eq:id_id}
	\end{equation}
	Then, for any $\Ubaldt \in \Lspace$, $\Vbaldt \in \Rspace$,
	we choose, in \eqref{eq:hessineq} (with $\lambda = 0$),
	$\Udt = \Ubaldt (V^\dagger \Vbal)\transpose$ and $\Vdt = \Vbaldt (U^\dagger \Ubal)\transpose$
	to obtain
	\begin{align*}
		0 &\leq \ip{\scrA^*(\scrA(M) - \y)}{ \Udt \Vdt \transpose } + \frac{1}{2} \norm{\scrA( U \Vdt\transpose + \Udt V\transpose )}^2 \\
		&= \ip{\scrA^*(\scrA(M) - \y)}{ \Ubaldt (V^\dagger \Vbal)\transpose  U^\dagger \Ubal \Vbaldt\transpose } \\
		&\qquad+ \frac{1}{2} \norm{\scrA( U U^\dagger \Ubal \Vbaldt\transpose + \Ubaldt (V V^\dagger \Vbal)\transpose )}^2 \\
		&= \ip{\scrA^*(\scrA(M) - \y)}{ \Ubaldt \Vbaldt\transpose } + \frac{1}{2} \norm{\scrA( \Ubal \Vbaldt\transpose + \Ubaldt \Vbal\transpose )}^2.
	\end{align*}
	The last equality follows from \eqref{eq:dag_id} and \eqref{eq:id_id}.
	This is precisely the Hessian inequality \eqref{eq:bal_hess}.
\end{proof}

\subsection{Key inequality from second-order criticality}
The following lemma is the core of our analysis;
this is what allows us to only require upper isometry of $\scrA$ on the subspace $\Tst$.
\begin{lemma}
	\label{lem:basic}
	Under the model \eqref{eq:ms_model},
	let $\Mst = \Ust \Vst\transpose$ for $\Ust \in \Lstspace$, $\Vst \in \Rstspace$ that are balanced, that is, $\Ust\transpose \Ust = \Vst\transpose \Vst$.
	Suppose $(U, V)$ is a balanced second-order critical point of \eqref{eq:lasso_fact},
	and set $M \coloneqq U V\transpose$.
	For any $Q \in \R^{\ropt \times \rgt}$ such that $Q\transpose Q = I_{\rgt}$,
	setting
	\[
		\Delta_U \coloneqq U Q - \Ust \quad \text{and} \quad \Delta_V \coloneqq V Q - \Vst,
	\]
	we have
	\begin{align*}
		3 \norm{\scrA(M - \Mst)}^2
		&\leq \frac{1}{2} \norm{\scrA(\Delta_U \Vst\transpose + \Ust \Delta_V\transpose)}^2 \\
		&\qquad + 5 \ip{\scrA^*(\xi)}{M - \Mst} - 2 \ip{\scrA^*(\xi)}{\Delta_U \Vst\transpose + \Ust \Delta_V\transpose} \\
		&\qquad + 5\lambda(\nucnorm{\Mst} - \nucnorm{M}) + 2 \lambda \nucnorm{\Delta_U \Vst\transpose + \Ust \Delta_V\transpose}.
	\end{align*}
	The same holds in the symmetric case where $d_1 = d_2 = d$, $\Ust = \Vst$, and $U = V$ is a second-order critical point of \eqref{eq:lasso_sym_fact}.
\end{lemma}
Unlike for \Cref{lem:balred}, it is unclear how this could be generalized to non-quadratic objective functions.
It is interesting to compare \Cref{lem:basic} to \cite[Lem.~7]{Ge2017}: in the symmetric case with $\xi = 0$ and $\lambda = 0$, they obtain
\[
	3 \norm{\scrA(M - \Mst)}^2 \leq \norm{\scrA(\Delta_U \Delta_U\transpose)}^2,
\]
whereas our result instead has the quantity $\frac{1}{2} \norm{\scrA(\Delta_U \Ust\transpose + \Ust \Delta_U\transpose)}^2$ on the right-hand side.
Although one might expect $\norm{\scrA(\Delta_U \Delta_U\transpose)}^2$ to be smaller in many situations (e.g., when $\Delta_U$ is small),
it can be challenging to upper bound for random $\scrA$
because, in general, $\Delta_U \Delta_U\transpose \notin \Tst$.
\begin{proof}[Proof of \Cref{lem:basic}]
	We will use, in the Hessian inequality \eqref{eq:hessineq},
	$\Udt = \Ust Q\transpose - U$ and $\Vdt = \Vst Q\transpose - V$.
	The rest of the proof is a series of straightforward but tedious computations.
	We start with a long list of useful identities.
	
	The fact that $(\Ust, \Vst)$ and $(U, V)$ are balanced implies that
	\begin{equation}
		\label{eq:id_nn_bal}
		\normF{\Ust}^2 = \normF{\Vst}^2 = \nucnorm{\Mst}, \quad \text{and} \quad \normF{U}^2 = \normF{V}^2 = \nucnorm{M}.
	\end{equation}
	Next, note that
	\begin{equation}
		\label{eq:key_id_1}
		\begin{aligned}
			U \Vdt\transpose + \Udt V\transpose
			&= U Q \Vst\transpose + \Ust Q\transpose V\transpose - 2 M \\
			&= \Delta_U \Vst\transpose + \Ust \Delta_V\transpose - 2(M - \Mst),
		\end{aligned}
	\end{equation}
	and
	\begin{equation}
		\label{eq:key_id_12}
		\Udt \Vdt\transpose = \Mst + M - \Ust(VQ)\transpose - U Q \Vst\transpose.
	\end{equation}
	From \eqref{eq:key_id_1}, we also have
	\begin{equation}
		\label{eq:key_id_2}
		\scrA(U \Vdt\transpose + \Udt V\transpose)
		= \scrA(\Delta_U \Vst\transpose + \Ust \Delta_V\transpose) - 2(\scrA(M) - \y) - 2\xi.
	\end{equation}
	Now, let $E \in \Mspace$ be any subgradient to $\nucnorm{\cdot}$ at the point $M$ (e.g., the polar factor of $M$).
	As $(U, V)$ is balanced, we must have
	\begin{equation}
		\label{eq:key_id_3}
		E V = U, \quad E\transpose U = V, \quad \text{and} \quad \ip{E}{M} = \nucnorm{M}.
	\end{equation}
	This together with \eqref{eq:key_id_1} implies
	\begin{align*}
		\ip{E}{U \Vdt\transpose + \Udt V\transpose}
		&= \ip{E}{\Delta_U \Vst\transpose + \Ust \Delta_V\transpose} - 2 \nucnorm{M} + 2 \ip{E}{\Mst}
	\end{align*}
	From \eqref{eq:key_id_12} and \eqref{eq:key_id_3}, we also have
	\begin{align*}
		\ip{E}{\Udt \Vdt\transpose}
		= \ip{E}{\Mst} + \nucnorm{M} - \ip{\Ust}{UQ} - \ip{\Vst}{VQ}.
	\end{align*}
	We can furthermore calculate, using \eqref{eq:id_nn_bal},
	\begin{align*}
		\frac{\normF{\Udt}^2 + \normF{\Vdt}^2}{2}
		&= \frac{\normF{ \Ust Q\transpose - U }^2 + \normF{\Vst Q\transpose - V}^2}{2} \\
		&= \frac{1}{2}( \normF{\Ust}^2 + \normF{U}^2 - 2 \ip{\Ust}{ U Q } + \normF{\Vst}^2 + \normF{V}^2 - 2\ip{\Vst}{VQ} ) \\
		&= \nucnorm{\Mst} + \nucnorm{M} - \ip{\Ust}{UQ} - \ip{\Vst}{VQ}.
	\end{align*}
	From these last three identities, we can estimate the quantity (which will appear later)%
	\begin{equation}
		\label{eq:key_id_4}
		\begin{aligned}
			\qquad&\negqquad \ip{E}{\Mst - M} + 2\ip{E}{U \Vdt\transpose + \Udt V\transpose} - \ip{E}{\Udt \Vdt\transpose} + \frac{ \normF{\Udt}^2 + \normF{\Vdt}^2 }{2} \\
			&= \nucnorm{\Mst} + 4 \ip{E}{\Mst} - 5 \nucnorm{M} + 2\ip{E}{\Delta_U \Vst\transpose + \Ust \Delta_V\transpose} \\
			&\leq 5 (\nucnorm{\Mst} - \nucnorm{M}) + 2 \nucnorm{\Delta_U \Vst\transpose + \Ust \Delta_V\transpose},
		\end{aligned}
	\end{equation}
	where the inequality follows from the fact that $\opnorm{E} = 1$.
	
	Now, we use the first-order condition \eqref{eq:gradzero}.
	This together with \eqref{eq:key_id_3} implies that, for any $L \in \Lspace$, $R \in \Rspace$,
	\begin{equation*}
		\ip{\scrA^*(\scrA(M) - \y) + \lambda E}{L V\transpose + U R\transpose} = 0.
	\end{equation*}
	In particular, we obtain
	\begin{equation}
		\label{eq:fo_lam_1}
		\ip{\scrA^*(\scrA(M) - \y) + \lambda E}{U \Vdt\transpose + \Udt V\transpose} = 0,
	\end{equation}
	then, by \eqref{eq:key_id_1},
	\begin{equation}\label{eq:fo_lam_2}
		\begin{aligned}
			\qquad&\negqquad \ip{\scrA^*(\scrA(M) - \y) + \lambda E}{ \Delta_U \Vst\transpose + \Ust \Delta_V\transpose } \\
			&= \ip{ \scrA^*(\scrA(M) - \y) + \lambda E }{2(M - \Mst)} \\
			&= 2 \norm{\scrA(M - \Mst)}^2 - 2 \ip{\scrA^*(\xi) - \lambda E}{M - \Mst},
		\end{aligned}
	\end{equation}
	and, finally, by \eqref{eq:key_id_12},
	\begin{equation}
		\label{eq:fo_lam_3}
		\begin{aligned}
			\ip{\scrA^*(\scrA(M) - \y) + \lambda E}{\Udt \Vdt\transpose}
			&= \ip{\scrA^*(\scrA(M) - \y) + \lambda E}{\Mst - M} \\
			&= - \norm{\scrA(M - \Mst)}^2 + \ip{\scrA^*(\xi) - \lambda E}{M - \Mst}.
		\end{aligned}
	\end{equation}
	We can then laboriously calculate (silently using \eqref{eq:key_id_1} and \eqref{eq:key_id_2} several times)
	\begin{align*}
		\qquad&\negqquad \frac{1}{2} \norm{\scrA( U \Vdt\transpose + \Udt V\transpose )}^2 \\
		&= \frac{1}{2} \ip{ \scrA(\Delta_U \Vst\transpose + \Ust \Delta_V\transpose) }{\scrA(U \Vdt\transpose + \Udt V\transpose)} \\
		&\qquad - \underbrace{\ip{\scrA^*(\scrA(M) - \y) + \lambda E}{U \Vdt\transpose + \Udt V\transpose}}_{= 0 \text{ by \eqref{eq:fo_lam_1}}} - \ip{\scrA^*(\xi) - \lambda E}{ U \Vdt\transpose + \Udt V\transpose } \\
		&= \frac{1}{2} \norm{\scrA(\Delta_U \Vst\transpose + \Ust \Delta_V\transpose)}^2 - \ip{ \Delta_U \Vst\transpose + \Ust \Delta_V\transpose }{\scrA^*(\scrA(M) - \y) + \lambda E} \\
		&\qquad - \ip{ \Delta_U \Vst\transpose + \Ust \Delta_V\transpose }{\scrA^*(\xi) - \lambda E}
		- \ip{\scrA^*(\xi) - \lambda E}{ U \Vdt\transpose + \Udt V\transpose } \\
		&\overset{\mathclap{\eqref{eq:fo_lam_2}}}{=} \frac{1}{2} \norm{\scrA(\Delta_U \Vst\transpose + \Ust \Delta_V\transpose)}^2 - 2 \norm{\scrA(M - \Mst)}^2 + 2 \ip{M - \Mst}{\scrA^*(\xi) - \lambda E} \\
		&\qquad - \ip{ \Delta_U \Vst\transpose + \Ust \Delta_V\transpose }{\scrA^*(\xi) - \lambda E}
		- \ip{\scrA^*(\xi) - \lambda E}{U \Vdt\transpose + \Udt V\transpose} \\
		&= \frac{1}{2} \norm{\scrA(\Delta_U \Vst\transpose + \Ust \Delta_V\transpose)}^2 - 2 \norm{\scrA(M - \Mst)}^2 - 2 \ip{\scrA^*(\xi) -\lambda E}{U \Vdt\transpose + \Udt V\transpose}.
	\end{align*}
	This together with the Hessian inequality \eqref{eq:hessineq} and \eqref{eq:key_id_4} and \eqref{eq:fo_lam_3} implies
	\begin{align*}
		0 &\leq \ip{\scrA^*(\scrA(M) - \y)}{\Udt \Vdt\transpose} + \frac{1}{2} \norm{\scrA( U \Vdt\transpose + \Udt V\transpose )}^2 + \frac{\lambda}{2} ( \normF{\Udt}^2 + \normF{\Vdt}^2) \\
		&= \ip{\scrA^*(\scrA(M) - \y) + \lambda E}{\Udt \Vdt\transpose} + \frac{1}{2} \norm{\scrA( U \Vdt\transpose + \Udt V\transpose )}^2 \\
		&\qquad- \lambda\ip{E}{\Udt \Vdt\transpose } + \frac{\lambda}{2} ( \normF{\Udt}^2 + \normF{\Vdt}^2) \\
		&= - 3 \norm{\scrA(M - \Mst)}^2 + \frac{1}{2} \norm{\scrA(\Delta_U \Vst\transpose + \Ust \Delta_V\transpose)}^2 \\
		&\qquad+ \ip{\scrA^*(\xi)}{M - \Mst - 2(U \Vdt\transpose + \Udt V\transpose)} \\
		&\qquad + \lambda \parens*{ \ip{E}{\Mst - M} - \ip{E}{\Udt \Vdt\transpose} + \frac{ \normF{\Udt}^2 + \normF{\Vdt}^2 }{2} + 2\ip{E}{U \Vdt\transpose + \Udt V\transpose} } \\
		&\leq - 3 \norm{\scrA(M - \Mst)}^2 + \frac{1}{2} \norm{\scrA(\Delta_U \Vst\transpose + \Ust \Delta_V\transpose)}^2 \\
		&\qquad + \ip{\scrA^*(\xi)}{5(M - \Mst) - 2(\Delta_U \Vst\transpose + \Ust \Delta_V\transpose)} \\
		&\qquad + \lambda\parens{5 (\nucnorm{\Mst} - \nucnorm{M}) + 2 \nucnorm{\Delta_U \Vst\transpose + \Ust \Delta_V\transpose}}.
	\end{align*}
	This completes the proof.
\end{proof}

\subsection{A norm inequality}
The following lemma will allow us to bound the quantity $\Delta_U \Vst\transpose + \Ust \Delta_V\transpose$ that appears in \Cref{lem:basic}:
\begin{lemma}
	\label{lem:normineq}
	Let $\Mst = \Ust \Vst\transpose$ for $\Ust \in \Lstspace$, $\Vst \in \Rstspace$ with $\Ust\transpose \Ust = \Vst\transpose \Vst$.
	Let $U \in \Lspace$, $V \in \Rspace$ be balanced (i.e., $U\transpose U = V\transpose V$),
	and set $M \coloneqq U V\transpose$.
	
	Let $Q \in \R^{\ropt \times \rgt}$ be the polar factor of $U \transpose \Ust + V\transpose 
	\Vst$.
	With $\Delta_U \coloneqq U Q - \Ust$ and $\Delta_V \coloneqq V Q - \Vst$ as in \Cref{lem:basic}, we have
	\[
		\normF{M - \Mst}^2 \geq \frac{1}{\sqrt{5} + 2} \normF{\Delta_U \Vst\transpose + \Ust \Delta_V\transpose}^2.
	\]
	The same holds in the symmetric case with $d_1 = d_2 = d$, $\Ust = \Vst$, and $U = V$.
\end{lemma}
\begin{proof}
	In contrast to the preceding proofs,
	we first consider the symmetric case;
	we will then analyze the \emph{a}symmetric case by a standard reduction.
	$Q$ is now a polar factor of $U\transpose \Ust$,
	and we set $\Delta \coloneqq U Q - \Ust$.
	With this choice of $Q$, $\Ust\transpose U Q \in \R^{\rgt \times \rgt}$ is symmetric and positive semidefinite.
	This further implies that $\Ust\transpose \Delta = \Ust \transpose U Q - \Ust\transpose \Ust$ is symmetric.
	
	Noting, in addition, that $\Ust\transpose U(I_\ropt - Q Q\transpose) = 0$,
	we can first bound
	\begin{align*}
		\normF{M - \Mst}^2
		&= \normF{U U\transpose - \Ust \Ust\transpose}^2 \\
		&= \normF{U Q Q\transpose U\transpose - \Ust \Ust\transpose + U (I_{\ropt} - Q Q\transpose) U\transpose}^2 \\
		&\geq \normF{U Q Q\transpose U\transpose - \Ust \Ust\transpose}^2 + 2 \underbrace{\ip{U Q Q\transpose U\transpose - \Ust \Ust\transpose}{U (I_{\ropt} - Q Q\transpose) U\transpose}}_{\geq 0} \\
		&\geq \normF{(\Ust + \Delta)(\Ust + \Delta)\transpose - \Ust \Ust\transpose}^2 \\
		&= \normF{\Delta \Ust\transpose + \Ust \Delta\transpose + \Delta \Delta\transpose}^2.
	\end{align*}
	We now use an argument inspired by (but distinct from) \cite[pf.\ of Lem.~41]{Ge2017}.
	Starting from the previous inequality, we can calculate
	\begin{align*}
		\normF{M - \Mst}^2
		&\geq \tr[ 2 \Ust\transpose \Ust \Delta\transpose \Delta + 2(\Ust\transpose \Delta)^2 + (\Delta \transpose \Delta)^2 + 4 \Ust\transpose \Delta \Delta\transpose \Delta ] \\
		&= (2 \sqrt{5} - 4) \underbrace{\tr[ \Ust\transpose \Ust \Delta\transpose \Delta + (\Ust\transpose \Delta)^2 ]}_{= \frac{1}{2} \normF{\Delta \Ust\transpose + \Ust \Delta\transpose}^2 } \\
		&\qquad + (6 - 2 \sqrt{5}) \underbrace{\tr[ \Ust\transpose (\Ust + \Delta) \Delta\transpose \Delta ]}_{= \ip{\Ust \transpose U Q}{\Delta \transpose \Delta} \geq 0}
		+ \underbrace{\tr[ (\Delta\transpose \Delta + (\sqrt{5} - 1) \Ust\transpose \Delta)^2 ]}_{= \normF{\Delta\transpose \Delta + (\sqrt{5} - 1) \Ust\transpose \Delta}^2 \geq 0} \\
		&\geq (\sqrt{5} - 2)\normF{\Delta \Ust\transpose + \Ust \Delta\transpose}^2.
	\end{align*}
	The claimed inequalities in the penultimate line hold due to the facts that $\Ust\transpose U Q \succeq 0$ and $\Ust\transpose \Delta$ is symmetric.
	Noting that $\sqrt{5} - 2 = \frac{1}{\sqrt{5} + 2}$, this completes the proof in the symmetric case.
	
	For the asymmetric case, rather than attempting the calculations directly,
	we use a standard reduction to the symmetric case.
	See, for example, \cite[Sec.~3.2]{Zhang2025a}, whose notation we adapt.
	For a symmetric block matrix of the form
	\[
		X = \begin{bmatrix*}
			X_{11} & X_{12} \\
			X_{12}\transpose & X_{22}
		\end{bmatrix*} \in \symms_{d_1 + d_2}
		\quad \text{for} \quad X_{11} \in \symms_{d_1},\ X_{12} \in \Mspace, \ X_{22} \in \symms_{d_2},
	\]
	we have
	\begin{equation}
		\label{eq:X_decomp}
		\begin{aligned}
			\normF{X_{12}}^2
			&= \frac{1}{4} ( \normF{X}^2 - \ip{X}{J X J} ),
			\quad \text{where} \quad J \coloneqq \begin{bmatrix*}
				I_{d_1} & 0_{d_1 \times d_2} \\
				0_{d_2 \times d_1} & -I_{d_2}
			\end{bmatrix*}.
		\end{aligned}
	\end{equation}
	Now, to study $\normF{M - \Mst}^2 = \normF{U V\transpose - \Ust \Vst\transpose}^2$,
	we set
	\[
		W \coloneqq \begin{bmatrix*}
			U \\ V
		\end{bmatrix*},
		\quad
		\Wst \coloneqq \begin{bmatrix*}
			\Ust \\ \Vst
		\end{bmatrix*},
		\quad \text{and} \quad
		\Delta_W \coloneqq W Q - \Wst.
	\]
	With $X \coloneqq W W\transpose - \Wst \Wst\transpose$, we have, from \eqref{eq:X_decomp},
	\begin{equation}
		\begin{aligned}
			\normF{M - \Mst}^2
			= \normF{X_{12}}^2
			&= \frac{1}{4} \normF{W W\transpose - \Wst \Wst\transpose}^2 \\
			&\qquad  - \frac{1}{4} \ip{W W\transpose - \Wst \Wst\transpose}{J(W W\transpose - \Wst \Wst\transpose)J} .
		\end{aligned}
		\label{eq:err_W_expr}
	\end{equation}
	We have assumed that $Q$ is a polar factor of $U\transpose \Ust + V\transpose \Vst = W\transpose \Wst$.
	Therefore, by the symmetric case (with $d = d_1 + d_2$), we have the bound
	\begin{equation}
		\normF{W W\transpose - \Wst \Wst\transpose}^2
		\geq \frac{1}{\sqrt{5} + 2} \normF{\Delta_W \Wst\transpose + \Wst \Delta_W\transpose}^2.
		\label{eq:err_W_lb}
	\end{equation}
	To calculate the inner product involving $J$ in \eqref{eq:err_W_expr},
	note that, because $(U, V)$ and $(\Ust, \Vst)$ are balanced,
	\[
		W\transpose J W = U\transpose U - V\transpose V = 0, \quad \text{and} \quad \Wst\transpose J \Wst = \Ust\transpose \Ust - \Vst\transpose \Vst = 0.
	\]
	Therefore
	\begin{align*}
		\ip{W W\transpose - \Wst \Wst\transpose}{J(W W\transpose - \Wst \Wst\transpose)J}
		&= - 2\ip{\Wst \Wst\transpose}{J W W\transpose J} \\
		&= -2 \normF{W\transpose J \Wst}^2 \\
		&= -2 \normF{U\transpose \Ust - V\transpose \Vst}^2.
	\end{align*}
	Plugging this and \eqref{eq:err_W_lb} into \eqref{eq:err_W_expr},
	we have
	\begin{equation}
		\label{eq:symm_lb_1}
		\normF{M - \Mst}^2 \geq \frac{1}{4(\sqrt{5} + 2)} \normF{\Delta_W \Wst\transpose + \Wst \Delta_W\transpose}^2 + \frac{1}{2} \normF{U\transpose \Ust - V\transpose \Vst}^2.
	\end{equation}
	To get the final result,
	we now need to estimate $\normF{\Delta_W \Wst\transpose + \Wst \Delta_W\transpose}^2$ in terms of $\normF{\Delta_U \Vst\transpose + \Ust \Delta_V\transpose}^2$.
	Taking, now, $X \coloneqq \Delta_W \Wst\transpose + \Wst \Delta_W\transpose$ in \eqref{eq:X_decomp},
	we have
	\begin{align*}
		\normF{\Delta_U \Vst\transpose + \Ust \Delta_V\transpose}^2
		&= \frac{1}{4} \normF{\Delta_W \Wst\transpose + \Wst \Delta_W\transpose}^2 \\
		&\qquad - \frac{1}{4} \ip{\Delta_W \Wst\transpose + \Wst \Delta_W\transpose}{J (\Delta_W \Wst\transpose + \Wst \Delta_W\transpose) J}.
	\end{align*}
	The fact that $\Wst \transpose J \Wst = 0$ implies
	\begin{align*}
		\ip{\Delta_W \Wst\transpose + \Wst \Delta_W\transpose}{J (\Delta_W \Wst\transpose + \Wst \Delta_W\transpose) J}
		&= 2 \tr[ (\Wst\transpose J \Delta_W)^2 ] \\
		&= 2 \tr[ (\Wst\transpose J W Q)^2 ] \\
		&= 2 \tr[ ( (\Ust\transpose U - \Vst\transpose V )Q )^2 ] \\
		&\geq -2 \normF{U\transpose \Ust - V\transpose \Vst}^2.
	\end{align*}
	We therefore have
	\begin{align*}
		\frac{1}{4} \normF{\Delta_W \Wst\transpose + \Wst \Delta_W\transpose}^2
		&= \normF{\Delta_U \Vst\transpose + \Ust \Delta_V\transpose}^2 \\
		&\qquad + \frac{1}{4} \ip{\Delta_W \Wst\transpose + \Wst \Delta_W\transpose}{J (\Delta_W \Wst\transpose + \Wst \Delta_W\transpose) J} \\
		&\geq \normF{\Delta_U \Vst\transpose + \Ust \Delta_V\transpose}^2
		- \frac{1}{2} \normF{U\transpose \Ust - V\transpose \Vst}^2.
	\end{align*}
	Plugging this into \eqref{eq:symm_lb_1} gives
	\[
		\normF{M - \Mst}^2 \geq \frac{1}{\sqrt{5} + 2} \normF{\Delta_U \Vst\transpose + \Ust \Delta_V\transpose}^2 + \frac{1}{2} \parens*{ 1 - \frac{1}{\sqrt{5} + 2}} \normF{U\transpose \Ust - V\transpose \Vst}^2.
	\]
	Dropping the (nonnegative) second term, we obtain the result in the asymmetric case.
\end{proof}

\subsection{Putting it together}
We can now prove \Cref{thm:determ} with the tools developed in the preceding subsections.
We write $\Mst = \Ust \Vst\transpose$ for some $\Ust \in \Lstspace$, $\Vst \in \Rstspace$ with $\Ust\transpose \Ust = \Vst\transpose \Vst$.
Let $(U, V)$ be a second-order critical point of \eqref{eq:lasso_fact},
and set $M \coloneqq U V\transpose$
(in the symmetric case, $\Ust = \Vst$, and $U = V$ is a second-order critical point of \eqref{eq:lasso_sym_fact}).
Let $H \coloneqq M - \Mst$.

By \Cref{lem:balred}, we can assume, without loss of generality, that $(U, V)$ is balanced.
As in \Cref{lem:normineq}, let $Q \in \R^{\ropt \times \rgt}$ be a polar factor of $U\transpose \Ust + V\transpose \Vst$.
As in \Cref{lem:basic,lem:normineq}, we set $\Delta_U \coloneqq U Q - \Ust$ and $\Delta_V \coloneqq V Q - \Vst$.
Then, noting that $\Delta_U \Vst\transpose + \Ust \Delta_V\transpose \in \Tst$,
our isometry assumptions on $\scrA$ together with \Cref{lem:normineq} give
\begin{align*}
	3 \norm{\scrA(H)}^2 - \frac{1}{2} \norm{\scrA(\Delta_U \Vst\transpose + \Ust \Delta_V\transpose)}^2
	&\geq 3 \alpha \normF{H}^2 - \frac{\beta}{2} \normF{ \Delta_U \Vst\transpose + \Ust \Delta_V\transpose }^2 \\
	&\geq \underbrace{\parens*{3 \alpha - \frac{\sqrt{5} + 2}{2} \beta }}_{\eqqcolon \mu} \normF{H}^2.
\end{align*}
Together with this inequality, \Cref{lem:basic} implies
\begin{equation*}
	\begin{aligned}
		\mu \normF{H}^2
		&\leq 5 \ip{\scrA^*(\xi)}{H} - 2 \ip{\scrA^*(\xi)}{\Delta_U \Vst\transpose + \Ust \Delta_V\transpose} \\
		&\qquad + 5\lambda(\nucnorm{\Mst} - \nucnorm{M}) + 2 \lambda \nucnorm{\Delta_U \Vst\transpose + \Ust \Delta_V\transpose} \\
		&\leq 5[ \opnorm{\scrA^*(\xi)} \nucnorm{H} + \lambda (\nucnorm{\Mst} - \nucnorm{M}) ]
		\\
		&\qquad+ 2 (\opnorm{\scrA^*(\xi)} + \lambda) \nucnorm{\Delta_U \Vst\transpose + \Ust \Delta_V\transpose}.
	\end{aligned}
\end{equation*}
Standard arguments based on nuclear norm geometry (see \cite[Sec.~5.1.1]{McRae2026preprint}) give
\begin{align*}
	&\opnorm{\scrA^*(\xi)} \nucnorm{H} + \lambda(\nucnorm{\Mst} - \nucnorm{M}) \\
	&\qquad\leq [(1 + \sqrt{2}) \sqrt{\rgt} \lambda + \sqrt{\ropt + \rgt} ( \opnorm{\scrA^*(\xi)} - \lambda )_+] \normF{H}.
\end{align*}
Furthermore, as $\Delta_U \Vst\transpose + \Ust \Delta_V\transpose \in \Tst$ has rank at most $2 \rgt$, we have (again using \Cref{lem:normineq})
\[
	\nucnorm{\Delta_U \Vst\transpose + \Ust \Delta_V\transpose}
	\leq \sqrt{2 \rgt} \normF{ \Delta_U \Vst\transpose + \Ust \Delta_V\transpose }
	\leq \sqrt{2 (\sqrt{5} + 2) \rgt } \normF{H}.
\]
Combining these last three displays, we obtain
\begin{align*}
	\mu\normF{H}
	&\leq \parens*{5(1 + \sqrt{2})\lambda + 2 \sqrt{2 (\sqrt{5} + 2)}(\opnorm{\scrA^*(\xi)} + \lambda) }\sqrt{\rgt} \normF{H} \\
	&\qquad + 5 \sqrt{\ropt + \rgt} ( \opnorm{\scrA^*(\xi)} - \lambda )_+ \normF{H}.
\end{align*}
Substituting in the value of $\mu$, rearranging, and (for simplicity) rounding the constants complete the proof of \Cref{thm:determ}.

\section{Matrix sensing with random measurements}
\label{sec:ms}
In this section, we apply our deterministic result \Cref{thm:determ} to the case where the matrices $A_1, \dots, A_n$ in \eqref{eq:ms_model} are chosen randomly.
Throughout the section, we will use $C, c$, etc.\ to denote positive constants that are universal (i.e., they do not depend on problem parameters) but which may change from one use to another.

\subsection{General random measurements}
We have the following general result, from which \Cref{thm:ms}, presented in \Cref{sec:intro}, will follow.
\begin{theorem}
	\label{thm:ms_gen}
	In the asymmetric case (resp.\ symmetric case with $d_1 = d_2 = d$), under the model \eqref{eq:ms_model},
	suppose $A_1, \dots, A_n$ are independent copies of a zero-mean random matrix $A \in \Mspace$ (resp.\ $A \in \symspace$) satisfying the following properties:
	\begin{enumerate}
		\item \label{assump:isotropic} (Isotropy) For all $H \in \Mspace$ (resp.\ $H \in \symspace$), $\E\abs{\ip{A}{H}}^2 = \normF{H}^2$.
		\item \label{assump:moment} (Moment condition) For some $q > 4$ and $B \geq 1$, we have, for all $H \in \Mspace$ (resp.\ $H \in \symspace$),
		\[
			(\E\abs{\ip{A}{H}}^q)^{1/q} \leq B \normF{H}.
		\]
		\item \label{assump:bdd} (Bounded on $\Tst$) For some $\tau > 0$, with probability at least $1 - \frac{1}{d_1 + d_2}$,
		\begin{align*}
			\max_i~\normF{\PTst(A_i)} &\leq \tau \sqrt{\rgt (d_1 + d_2)},
		\end{align*}
		where $\PTst$ is the orthogonal projection onto $\Tst$ in $\Mspace$ (resp. $\symspace$).
	\end{enumerate}
	There is $C_q > 0$ depending only on $q$ such that, for any $\ropt \geq \rgt$, if
	\[
		n \geq C_q [ \tau^2 \rgt + B^4 \ropt \log (d_1 + d_2) ] (d_1 + d_2),
	\]
	then, with probability at least $1 - \frac{C}{d_1 + d_2}$, \socptext{}
	\[
		\normF{M - \Mst} \leq \frac{C}{n} \errbdcomps.
	\]
\end{theorem}
We prove this in \Cref{sec:proofs_ms_gen} below.
The $\frac{C}{d_1 + d_2}$ failure probability bound comes from existing probabilistic tools that we use (see \Cref{lem:conc_tikh} and subsequent discussion below);
this could likely be improved (perhaps requiring a larger value of $q$), but we do not pursue this here.

\subsection{Proof for rank-1 measurements}
In this subsection, we use the general result \Cref{thm:ms_gen} to prove \Cref{thm:ms} for rank-1 random measurements. First, we need a technical lemma:
\begin{lemma}[{\cite[Thm.~2.1]{Hsu2012a}}]
	\label{lem:hw_ub}
	Let $X \in \R^D$ be a zero-mean $K$--sub-Gaussian random vector.
	For any fixed $D \times D$ positive semidefinite matrix $\Sigma$, we have, with probability at least $1 - e^{-t}$,
	\[
		\ip{\Sigma X}{X} \leq C K^2 \parens*{ \tr \Sigma + \sqrt{\tr(\Sigma^2) t} + \opnorm{\Sigma} t }.
	\]
\end{lemma}
In particular, taking $\Sigma = I_D$ (see also \cite{Liu2025}), we obtain, with probability at least $1 - e^{-t}$,
\[
	\norm{X}^2 \leq C K^2 (D + t).
\]
The definition of sub-Gaussianity in \cite{Hsu2012a} is slightly different from ours but is equivalent up to a constant.

\begin{proof}[Proof of \Cref{thm:ms}]
	We will naturally seek to apply \Cref{thm:ms_gen} with $A = a b\transpose$.
	Therefore, we must establish the conditions of \Cref{thm:ms_gen} for such a choice of $A$.
	\begin{enumerate}
		\item[\ref{assump:isotropic}.] (Isotropy) The isotropy of $A = a b\transpose$ follows easily from the fact that $a$ and $b$ are isotropic random vectors independent of each other.
		\item[\ref{assump:moment}.] (Moment condition) Fix $H \in \Mspace$.
		As $a$ is $K_a$--sub-Gaussian and independent of $b$, note that, for any integer $m \geq 1$, conditioned on $b$, we have
		\[
			\E_{a} \brackets*{ 1 + \frac{\abs{\ip{a}{Hb}}^{2m}}{ m! K_a^{2m} \norm{Hb}^{2m}} }
			\leq \E_{a} \brackets*{ \exp\parens*{ \frac{\abs{\ip{a}{Hb}}^2}{K_a^2 \norm{Hb}^2} } }\leq 2,
		\]
		which implies
		\begin{equation}
			\label{eq:a_exp}
			\E_{a} \abs{\ip{a b\transpose}{H}}^{2m} \leq m! K_a^{2m} \norm{Hb}^{2m}.
		\end{equation}
		We then apply \Cref{lem:hw_ub} with $\Sigma = H\transpose H$, noting that
		\[
			\opnorm{\Sigma} \leq \sqrt{\tr(\Sigma^2)} \leq \tr \Sigma = \normF{H}^2,
		\]
		to obtain, for any $t \geq 0$, with probability at least $1 - e^{-t}$,
		\[
			\norm{Hb}^2 \leq C K_b^2(1 + t) \normF{H}^2.
		\]
		This further implies (e.g., by standard tail integration arguments) that, for any $m \geq 1$,
		\[
			\E \norm{Hb}^{2m} \leq C m! K_b^{2m} \normF{H}^{2m}.
		\]
		Combining this with \eqref{eq:a_exp}, we obtain
		\[
			( \E \abs{\ip{a b\transpose}{H}}^{2m} )^{1/2m}
			\leq C (m!)^{1/m} K_a K_b \normF{H}.
		\]
		We can then choose, for example, $m = 4$ to obtain condition~\ref{assump:moment} of \Cref{thm:ms_gen} with $q \coloneqq 8$ and $B \coloneqq C K_a K_b$.
		
		\item[\ref{assump:bdd}.] (Bounded on $\Tst$)
		\newcommand{\QUst}{Q_{\Ust}}
		\newcommand{\QVst}{Q_{\Vst}}
		Let $\QUst \in \Lstspace$, $\QVst \in \Rstspace$ be matrices with orthonormal columns such that $\range(\QUst) = \range(\Ust)$ and $\range(\QUst) = \range(\Vst)$.
		The orthogonal projection onto $\Tst$ is given by the well-known formula
		\begin{align*}
			\PTst(H)
			=  \QUst \QUst\transpose H \parens*{I_{d_2} - \frac{1}{2} \QVst \QVst\transpose } + \parens*{I_{d_1} - \frac{1}{2}\QUst \QUst\transpose} H \QVst \QVst\transpose.
		\end{align*}
		We then have
		\begin{align*}
			\normF{\PTst(A_i)}
			&= \normF{\PTst(a_i b_i\transpose)} \\
			&\leq \norm{\QUst\transpose a_i} \norm{b_i} + \norm{a_i}\norm{\QVst\transpose b_i}.
		\end{align*}
		Noting that $\QUst\transpose a_i$ and $\QVst\transpose b_i$ are, respectively, $K_a$-- and $K_b$--sub-Gaussian random vectors in $\R^{\rgt}$,
		we can apply \Cref{lem:hw_ub} (in particular, the subsequent norm bound) to each norm in the above expression and for each $i$ to obtain, for any $t \geq 0$, with probability at least $1 - 4n e^{-t}$,
		\[
			\max_i~\normF{\PTst(A_i)} \leq C K_a K_b \sqrt{(\rgt + t)(d_1 + d_2 + t)}.
		\]
		Taking $t = \log 4 n(d_1 + d_2) \leq C \log n$ (recalling that $n \geq C(d_1 + d_2)$),
		we obtain condition~\ref{assump:bdd} of \Cref{thm:ms_gen} with
		\[
			\tau \coloneqq C K_a K_b \sqrt{ \parens*{1 + \frac{\log n}{\rgt}} \parens*{1 + \frac{\log n}{d_1 + d_2}} }.
		\]
	\end{enumerate}
	With this, we see that the condition for $n$ in \Cref{thm:ms_gen} is satisfied when (noting that $K_a, K_b \geq 1$ due to isotropy)
	\[
		n \geq C K_a^4 K_b^4 \ropt(d_1 + d_2) \log (d_1 + d_2),
	\]
	as we have assumed.
	The result follows from \Cref{thm:ms_gen}.
\end{proof}

\subsection{Proof of general result}
\label{sec:proofs_ms_gen}
In this section, we prove \Cref{thm:ms_gen} from the deterministic \Cref{thm:determ}.
It will follow easily from the following two technical lemmas whose proofs we defer to later in the section.
We only explicitly consider the asymmetric case, but everything is also valid in the symmetric case with $d_1 = d_2 = d$ and $\Mspace$ replaced by $\symspace$.
\begin{lemma}
	\label{lem:ms_lb}
	Under conditions \ref{assump:isotropic} and \ref{assump:moment} of \Cref{thm:ms_gen},
	for $\delta > 0$ and integer $k \geq 1$, if
	\[
		n \geq C \frac{ B^4 k (d_1 + d_2)}{\delta^2} \log \frac{d_1 + d_2}{\delta},
	\]
	then, with probability at least $1 - \frac{1}{d_1 + d_2}$,
	for all $H \in \Mspace$ of rank at most $k$,
	\[
		\frac{1}{n} \norm{\scrA(H)}^2 \geq (1 - \delta) \normF{H}^2.
	\]
\end{lemma}
The proof is in \Cref{sec:proofs_lb} below.

\begin{lemma}
	\label{lem:ms_tansp_ub}
	Under the assumptions of \Cref{thm:ms_gen},
	for any $\delta > 0$, if
	\[
		n \geq C_q \parens*{ \frac{\tau^2}{\delta} + \frac{B^4}{\delta^2} } \rgt (d_1 + d_2),
	\]
	where $C_q > 0$ only depends on $q$,
	then, with probability at least $1 - \frac{C}{d_1 + d_2}$,
	for all $H \in \Tst$,
	\[
		\frac{1}{n} \norm{\scrA(H)}^2 \leq (1 + \delta) \normF{H}^2.
	\]
\end{lemma}
The proof is in \Cref{sec:proofs_tansp_ub} below.

\begin{proof}[Proof of \Cref{thm:ms_gen}]
	Fix, independently of all other parameters, $\delta > 0$ such that
	\[
		\frac{1 + \delta}{1 - \delta} < \frac{6}{\sqrt{5} + 2} \quad \Longleftrightarrow \quad \delta < \frac{4 - \sqrt{5}}{8 + \sqrt{5}} \approx 0.17.
	\]
	The assumption on $n$ in the theorem statement allows us to apply \Cref{lem:ms_lb,lem:ms_tansp_ub} with the given (constant) $\delta$ and $k = \ropt + \rgt$.
	Then, with probability at least $1 - \frac{C}{d_1 + d_2}$,
	the assumptions of \Cref{thm:determ} are satisfied with
	\[
		\alpha \coloneqq n(1-\delta) \quad \text{and} \quad \beta \coloneqq n(1 + \delta). 
	\]
	The result follows immediately from the conclusion of \Cref{thm:determ} with some simplification and consolidation of terms.
\end{proof}

\subsubsection{Restricted lower isometry}
\label{sec:proofs_lb}
In this subsection, we prove the lower concentration bound \Cref{lem:ms_lb}.
The proof is inspired by that of \cite[Lem.~22]{Sun2018}.

The following known concentration result on sums of on nonnegative random variables is key:
\begin{lemma}{(E.g., \cite[Thm.~7.1]{Tropp2026})}
	\label{lem:nn_lower_tail}
	Let $X_1, \dots, X_n$ be independent and nonnegative real random variables, and let
	\[
		Y \coloneqq \sum_{i=1}^n X_i, \qquad v \coloneqq \sum_{i=1}^n \E X_i^2.
	\]
	Then, for all $t \geq 0$, we have, with probability at least $1 - e^{-t}$,
	\[
		Y \geq \E Y - \sqrt{2 v t}.
	\]
\end{lemma}
\begin{proof}[Proof of \Cref{lem:ms_lb}]
	\newcommand{\unitmats}{\scrS_{k}}
	\newcommand{\Neps}{\scrN_{\epsilon}}
	\newcommand{\neps}{n_{\epsilon}}
	We can, without loss of generality, restrict ourselves to matrices $H$ with $\normF{H} = 1$.
	For brevity, denote
	\[
		\unitmats \coloneqq \{ H \in \Mspace : \normF{H} = 1,\ \rank(H) \leq k \}.
	\]
	For fixed $H \in \unitmats$,
	we apply \Cref{lem:nn_lower_tail} with $X_i \coloneqq \abs{\ip{A_i}{H}}^2$.
	By isotropy, $\E X_i = \normF{H}^2 = 1$,
	and, by the moment assumption and Jensen's inequality,
	\[
		\E X_i^2 = \E \abs{\ip{A_i}{H}}^4 \leq (\E \abs{\ip{A_i}{H}}^q)^{4/q} \leq B^4.
	\]
	We thus have, in \Cref{lem:nn_lower_tail}, $v \leq B^4 n$,
	so, for all $t \geq 0$, with probability at least $1 - e^{-t}$,
	\[
		\frac{1}{n} \norm{\scrA(H)}^2 = \frac{1}{n} \sum_{i=1}^n X_i \geq 1 - C B^2 \sqrt{\frac{t}{n}}.
	\]
	This bound is for \emph{fixed} $H \in \unitmats$;
	to extend it to all $H \in \unitmats$, let $\Neps$ be an $\epsilon$-covering (in Frobenius norm) of $\unitmats$ for some $\epsilon > 0$ that we will specify later.
	By a union bound, we have, for all $t \geq 0$ (which we shift by $\log \abs{\Neps}$), with probability at least $1 - e^{-t}$,
	\[
		\frac{1}{n} \norm{\scrA(H)}^2 \geq 1 - C B^2 \sqrt{\frac{t + \log \abs{\Neps}}{n}} \quad \text{for all} \quad H \in \Neps.
	\]
	Next, for each $H \in \unitmats$,
	let $\neps(H) \in \Neps$ be a point satisfying $\normF{H - \neps(H)} \leq \epsilon$.
	We then have, for all $H \in \unitmats$,
	\begin{align*}
		\frac{1}{\sqrt{n}}\norm{\scrA(H)}
		&\geq \frac{1}{\sqrt{n}}\norm{\scrA(\neps(H))} - \frac{1}{\sqrt{n}}\norm{\scrA(H - \neps(H))} \\
		&\geq 1 - C B^2 \sqrt{\frac{t + \log \abs{\Neps}}{n}} - \frac{\opnorm{\scrA}}{\sqrt{n}} \epsilon.
	\end{align*}
	We will take $t = \log(2(d_1 + d_2))$ such that this holds with probability at least $1 - \frac{1}{2(d_1 + d_2)}$.
	We will see that $t$ is dominated by our bound on $\log \abs{\Neps}$,
	so we can drop $t$ modulo a constant.
	
	We can crudely upper bound $\opnorm{\scrA}$ by $\normF{\scrA}$.
	We have
	\[
		\E \normF{\scrA}^2 = \sum_{i=1}^n \E \normF{A_i}^2 = n d_1 d_2
	\]
	due to the isotropy of $A$.
	Then, by Markov's inequality, with probability at least $1 - \frac{1}{2(d_1 + d_2)}$, $\opnorm{\scrA} \leq \sqrt{2 n d_1 d_2 (d_1 + d_2)}$.
	Choosing $\epsilon = \frac{\delta}{4\sqrt{2d_1 d_2 (d_1 + d_2)}}$, we then have $\frac{\opnorm{\scrA}}{\sqrt{n}} \epsilon \leq \frac{\delta}{4}$.
	
	A minimal $\epsilon$-covering $\Neps$ of $\unitmats$ satisfies \cite[Lem.~3.1]{Candes2011b}
	\[
		\log \abs{\Neps} \leq k(d_1 + d_2 + 1) \log \frac{9}{\epsilon}
		\leq C k(d_1 + d_2) \log \frac{d_1 + d_2}{\delta}.
	\]
	If
	\[
		n \geq \frac{C B^4 k(d_1 + d_2)}{\delta^2} \log \frac{d_1 + d_2}{\delta}
	\]
	for sufficiently large $C$, then
	\[
		\frac{1}{\sqrt{n}}\norm{\scrA(H)} \geq 1 - \frac{\delta}{2} \quad \Longrightarrow \quad \frac{1}{n} \norm{\scrA(H)}^2 \geq 1 - \delta.
	\]
	Together with a union bound for the probability, this completes the proof.
\end{proof}

\subsubsection{Tangent space upper isometry}
\label{sec:proofs_tansp_ub}
In this subsection, we prove \Cref{lem:ms_tansp_ub}, the upper concentration bound of $\scrA$ restricted to $\Tst$.
It will follow easily from the following key probability result:
\begin{lemma}[{\cite[Cor.~1.2]{Tikhomirov2018}}]
	\label{lem:conc_tikh}
	Let $D$ and $n \geq 2 D$ be positive integers.
	Let $X_1, \dots, X_n$ be independent copies of a zero-mean isotropic random vector $X$ in $\R^D$ that satisfies the following:
	for some $q > 4$ and $B \geq 1$, for all unit-norm $h \in \R^D$,
	\[
		(\E \abs{\ip{X}{h}}^q)^{1/q} \leq B.
	\]
	Then, with probability at least $1 - \frac{1}{D}$,
	\[
		\opnorm*{ \frac{1}{n} \sum_{i=1}^n X_i X_i\transpose - I_D} \leq C_q \parens*{\frac{\max_i~\norm{X_i}^2}{n} + B^2 \sqrt{\frac{D}{n}}},
	\]
	where $C_q > 0$ is a constant only depending on $q$.
\end{lemma}
The $1/D$ bound on the failure probability is the bottleneck for our probability bounds in this paper.
We believe that it could be improved (perhaps with stronger assumptions on $q$), but we do not investigate this at present.
In the case of the rank-1 measurements of \Cref{thm:ms}, we could instead use results for sub-exponential random vectors (e.g., those of \cite{Zhivotovskiy2024} with an extra truncation argument) to obtain a smaller failure probability bound.

\begin{proof}[Proof of \Cref{lem:ms_tansp_ub}]
	We use \Cref{lem:conc_tikh} with $D = \dim(\Tst)$, the identification of $\Tst$ with $\R^D$,
	and $X = \PTst(A)$.
	Note that $D \leq \rgt(d_1 + d_2)$, and, by assumption, with probability at least $1 - \frac{1}{d_1 + d_2}$,
	\[
		\max_i~\norm{X_i} = \max_i~\normF{\PTst(A_i)} \leq \tau \sqrt{\rgt(d_1 + d_2)}.
	\]
	We can furthermore take the $q$ and $B$ in \Cref{lem:conc_tikh} to be the same as in the moment condition of \Cref{thm:ms_gen}.
	Therefore, with probability (by a union bound, noting that $D \geq c (d_1 + d_2)$) at least $1 - \frac{C}{d_1 + d_2}$,
	we have
	\begin{align*}
		\opnorm*{ \frac{1}{n} \PTst \scrA^* \scrA \PTst - \scrI_{\Tst} }
		\leq C_q\parens*{ \frac{\tau^2 \rgt (d_1 + d_2) }{n} + B^2 \sqrt{\frac{\rgt(d_1 + d_2)}{n}} }.
	\end{align*}
	The result immediately follows.
\end{proof}


\bibliographystyle{siamplain}
\bibliography{refs}

\end{document}